\documentclass[11pt,onecolumn]{article}
\usepackage[a4paper, hmargin={2.8cm, 2.8cm}, vmargin={2.5cm, 2.5cm}]{geometry}

\usepackage[utf8]{inputenc}
\usepackage{tikz}

\usepackage{amsmath}
\usepackage{amssymb}  
\usepackage{amsthm}
\usepackage{amsfonts}
\usepackage{graphicx}
\usepackage{cancel}
\usepackage{bbm}
\usepackage{url}
\usepackage{enumerate} 
\usepackage{ upgreek }
\usepackage{dsfont}
\usepackage{color}
\usepackage{subcaption}
\usepackage{comment}
\usepackage{makecell}
\usepackage{diagbox}
\usepackage{wrapfig}
\usepackage{rotating}
\usepackage{tabularx}

\usepackage{hyperref}
\hypersetup{
    colorlinks=true,
    linkcolor=blue,
    citecolor=red,
    filecolor=magenta,      
    urlcolor=cyan,
}

\renewcommand{\d}[1]{\ensuremath{\textnormal{d}#1}}

\newcommand{\cE}{\mathcal{E}}
\newcommand{\cF}{\mathcal{F}}

\def\beq{\begin{equation}}
\def\eeq{\end{equation}}
\def\bq{\begin{quote}}
\def\eq{\end{quote}}
\def\ben{\begin{enumerate}}
\def\een{\end{enumerate}}
\def\bit{\begin{itemize}}
\def\eit{\end{itemize}}

\def\r|{\right|}

\newcommand\Z{\mathbbm{Z}}
\newcommand\R{\mathbbm{R}}

\theoremstyle{plain}
\newtheorem{thm}{Theorem}[section]
\newtheorem{lem}[thm]{Lemma}
\newtheorem{cor}[thm]{Corollary}

\newtheorem{defi}[thm]{Definition}

\theoremstyle{definition}

\newtheorem{rem}{Remark}[section]

\renewcommand{\leq}{\leqslant}
\renewcommand{\geq}{\geqslant}

\title{Stability for the logarithmic Hardy-Littlewood-Sobolev Inequality with application to the Keller-Segel equation}
\author{Eric A Carlen\footnote{Partially supported by NSF grant DMS-2055282}\\ Department of Mathematics, Rutgers University\\ Hill Center for the Mathematical Sciences, 110 Freylinghuysen Rd.\\
Piscataway NJ 08854-8019, USA}

\begin{document}

\maketitle

\begin{abstract}
We apply a duality method to prove an optimal stability theorem for the logarithmic Hardy-Littlewood-Sobolev inequality, and we apply it to the estimation of the rate of approach to equilibrium for the critical mass Keller-Segel system.  
\end{abstract}

\section{Introduction}

The logarithmic Hardy-Littlewood-Sobolev (log HLS) inequality is a limiting case of the Hardy-Littlewood-Sobolev inequality established in its sharp form
 by Lieb \cite{L83}. 

For a non-negative measurable function $f$  on $\R^d$, $d\ge 1$, and a number
$\lambda$ with 
$0 \le \lambda < d$, the quadraric functional  $I_\lambda(f)$ given by
\begin{equation*}
I_\lambda(f) := \int_{\R^d}\int_{\R^d} f(x)\frac{1}{|x-x'|^\lambda}f(x')\d x \d x' \ ,
\end{equation*}
 is well defined, though it may take the value $+\infty$.   
For $1\le p < \infty$, let $L^p(\R^d)$ denote the corresponding Lebesgue space with  
$\|g\|_p$ denoting the  norm of $g\in L^p(\R^d)$.

A fundamental theorem of Hardy, Littlewood and Sobolev   provides an upper bound on $I_\lambda(f)$ in terms of 
$\|f\|_{p(\lambda)}$ where 
\begin{equation*}
p(\lambda) = \frac{2d}{2d- \lambda}\ .
\end{equation*}
That is, there is some finite $C_{\lambda,d}$ constant such that  $I_\lambda(f) \leq C_{\lambda,d}\|f\|_{p(\lambda)}^2$ for all $f\in L^p(\R^d)_{p(\lambda)}$.
Lieb's Theorem  \cite{L83} gives the best constant and all of  the cases of equality:

\begin{thm}[Lieb, 1983]\label{shls}
For  all $\lambda>0$ and all non-negative measurable functions $f$
with $\|f\|_{p(\lambda)} < \infty$, 
\begin{equation*}
I_\lambda(f) \le   \left(\frac{I_\lambda(h,h)}{\|h\|_{p(\lambda)}^2}\right)\|f\|_{p(\lambda)}^2
 \end{equation*}
where
${\displaystyle h(x) = \left(\frac{1}{1+|x|^2}\right)^{d/p}}$.
There is equality if and only if for some $x_0\in \R^d$ and $s\in \R_+$, $f$ is a non-zero multiple of $h(x/s-x_0)$ for some $s>0$ and some $x_0\in \R^d$. 
\end{thm}

The cases in which $\lambda = d-2$ are particularly important since for $d\geq 3$,
\begin{equation*}
\langle f, (-\Delta)^{-1} f\rangle _{L^2(\R^d)} = \frac{1}{|S^{d-1}|}\int_{\R^d}\int_{\R^d} f(x)\frac{1}{|x-y|^{d-2}}f(y)\d x\d y\ ,
\end{equation*}
where $|S^{d-1}|$ denotes the $d-2$ dimensional volume of the unit sphere in $\R^d$. 
Hence, up to a constant, $I_{d-2}(f,f)$ is the Newtonian potential of $f$.  
For $d=2$, there is an interesting and somewhat subtle limiting case of the HLS inequality, known as the log HLS Inequality.  As is well known, 
if $f\in C_0^\infty(\R^2)$,
and $u$ defined by
\begin{equation}\label{uf}
u(x) = \frac{1}{2\pi}\int_{\R^2} \ln |x-y| f(y)\d y\ ,
\end{equation}
then $\Delta u(x) = f(x)$.
On this basis, one might be tempted to conclude that 
\begin{equation}\label{uf2}
-\frac{1}{4\pi}\int_{\R^2}f(x)\ln|x-y|f(y)\d x \d y = \frac{1}{2}\langle f,(-\Delta)^{-1}f\rangle_{L^2(\R^2)} = \frac12\int_{\R^d}|\nabla u(x)|^2{\rm d}x
\end{equation}
for all $f$ in some dense subset of $L^1(\R^2)$, but this is not correct. 
When $u(x)$ is given by \eqref{uf}, then $\lim_{x\to\infty} |x||\nabla u(x)| = \int_{\R^2}f(x)\d x$,
and hence $|\nabla u(x)|^2$ is not integrable when $\int_{\R^d}f(x){\rm d}x \neq 0$.  Worse still, since $\ln|x-y|$ is unbounded both above and below,  the integral on the left in \eqref{uf2} is not well-defined for all non-negative $f\in L^1(\R^2)$. 

However, for $\lambda = 0$ and all $d$, $p(0) =1$, and the HLS inequality reduces to the identity $I_0(f,f) = \|f\|_1^2$, so that with $C_{g,\lambda}$ denoting Lieb's sharp constant,
$$\frac{I_\lambda(f) -I_0(f,f)}{\lambda} \leq \frac{C_{\lambda,d}\|f\|_{p(\lambda)}^2 - \|f\|_1^2}{\lambda} = \frac{C_{\lambda,d} -1}{\lambda}\|f\|_{p(\lambda)}^2 + \frac{\|f\|_{p(\lambda)}^2 - \|f\|_1^2}{\lambda}\\ .$$
Defining $C'_d = \frac{\d }{\d \lambda}C_{\lambda,d}\bigg|_{\lambda = 0}$ and formally taking the limit $\lambda\downarrow 0$ yields
\begin{equation*}
\left[
\int_{\R^d} f\log f \d x - \|f\|_1 \log \|f\|_1\right]  +
\frac{d}{\|f\|_1}\int_{\R^d}\int_{\R^d}f(x)\log(|x-y|)f(y)\d x \d y \ge -  dC'_d\|f\|_1\ .
\end{equation*}
At a formal level, this is the log HLS inequality. However, for $f\in L^1(\R^d)$, even if $f \geq 0$, neither side of this inequality is well defined:    $f\log f$ need not be integrable, and the logarithmic kernel is unbounded above and below.
However, for $f\in L^1(\R^d,(e+|x|^2)\d x)$ both sides are well defined \cite{CL92}, and then for $d=2$ one has the following inequality for the Newtonian potential:

\begin{thm}[Logarithmic Hardy-Littlewood-Sobolev inequality]\label{lhlsX}
Let $\rho$ be a non-negative
function in  $L^1(\R^2, \ln(e+|x|^2)\d x))$ such that $\int_{\R^2}\rho\d x = 1$.
Define
\begin{equation}\label{lhls1}
\mathcal{H}(\rho) :=  \int_{\R^2}\rho(x) \ln \rho(x) \d x + 2\int_{\R^d}\int_{\R^d}\rho(x)\ln(|x-x'|)\rho(y)\d x \d x'    + 1+\log \pi\ . 
\end{equation}
Then 
\begin{equation}\label{LHLSE}
\mathcal{H}(\rho) \geq 0\ ,
\end{equation}
 and there is equality  if and only if for some $x_0\in \R^d$ and $s>0$,
 \begin{equation}\label{hdef}
\rho(x) = s^{-2}h(x/s-x_0) \quad{\rm where}\quad h(x) = \frac{1}{\pi}\left(\frac{1}{1+|x|^2 }\right)^2\ .
\end{equation}
\end{thm}

This inequality was proved independently and at about the same time by myself and Loss \cite{CL92} and Beckner \cite{B93}. 
The proof in \cite{CL92} is based on the method of ``competing symmetries'' introduced in \cite{CL90} which makes use of the conformal 
invariance of the HLS inequalities, and the log HLS inequality by extension, to both prove existence of optimizers and to determine all of them. 
The conformal invariance of the HLS inequalities was discovered by Lieb \cite{L83} and was used in his paper to determine all of the 
optimizers, once they had been shown to exist. In both \cite{CL92} and \cite{B93} one finds versions 
of the log HLS inequality in all dimensions, not only $d=2$. 
However, in this paper we are concerned with $d=2$ due to the connection with the Newtonian potential, and for simplicity have stated only the relevant case. 

The notation $\mathcal{H}(\rho)$ used in Theorem~\ref{lhlsX} is meant to suggest the Helmholtz free energy which is $U -ST$ where $U$ is the internal energy, $S$ is the entropy and $T$ is the temperature. We may think of $ -\int_{\R^2}\rho(x) \ln f(x) \d x $ as representing the entropy $S$ of $\rho$, and take $T =1$, and then identify the remaining terms in $\mathcal{H}(\rho)$ as $U$, the internal energy associated to $f$.

We obtain an inequality for all non-negative  functions $\rho \in L^1(\R^2, \ln(e+|x|^2)\d x))$ not identically zero by defining $M := \int_{\R^2} \rho \d x$, and then applying the inequality to $\rho/M$.  

The main result of this paper is a stability bound for the log HLS inequality.  Define $\mathcal{M}$ to be the ``manifold'' of optimizers for the log HLS inequality:
\begin{equation}\label{ManDEF}
\mathcal{M} := \{ \ s^{-2}h(x/s-x_0)  \:\ s>0\ , x_0\in \R^2\ \}
\end{equation}
with $h$ given by \eqref{hdef}. We shall show that if  $\mathcal{H}(f)$, then the $L^1$ distance to $\mathcal{M}$ is small with a very explicit bound relating these two quantities:

\begin{thm}\label{main} Let $\rho$ be a non-negative
function in  $L^1(\R^2, \ln(e+|x|^2)\d x))$ such that $\int_{\R^2}\rho\d x = 1$.
Then
\begin{equation}\label{mainINEQ}
\mathcal{H}(\rho) \geq \frac{1}{8}\inf_{g\in \mathcal{M}}\|\rho -g\|_1^2\ .
\end{equation}
\end{thm}

This improves considerably on a result proved in \cite{CF}, where it was shown that for some finite constant  
$C$, and with additional conditions on $f$, 
\begin{equation}\label{ACLHLSS}
\mathcal{H}(f) \geq C\inf_{g\in \mathcal{M}}\|\rho -g\|_1^{20}\ .
\end{equation}

\subsection{Application to the critical mass Keller-Segel equation}

In this subsection we briefly explain an application of Theorem~\ref{main}, and hope to convey an understanding of how such stability results can be applied to problems such as determining the rates of approach to equilibrium.

The inequality \eqref{ACLHLSS} was proved was proved and applied in  \cite{CF} to quantify the rate of convergence to 
equilibrium in the critical mass Keller-Segel equation.   
The Keller-Segal equation  is a biological model describing {\em chemotaxis}: Consider a colony of bacteria on a large (infinite) Petri dish, 
across which they diffuse, and as they move, they emit a chemical attractant which then also diffuses. Let  $\rho(x)$ denote the density of the 
bacteria at $x$, and let $M := \int_{\R^2}\rho\d x$ denote the total mass of the bacteria.  The Keller-Segal equation 
describes the evolution of the density $\rho$, incorporating terms representing the concentrating effects of  the attractant, and diffusive effects 
representing the random component of the motion of the bacteria.  
It was show by Dolbeault and Perthame \cite{DP04} (see also \cite{BDP06}) that there is a critical mass, namely $8\pi$, such that for 
$M < 8\pi$, the diffusion wins, and the bacteria diffuse away to infinity, while if $M > 8\pi$ the concentration wins, and in finite time the solution develops a singularity as all of the mass collapses around a single point, corresponding to something actually observed in nature. 

The key to the proof that $8\pi$ is the critical mass is that the Keller Segal system can be written as  gradient flow for a 
close relative of the functional $\mathcal{H}$ specified in \eqref{lhls1} with respect to the $2$-Wasserstein distance, and then the sharp HLS inequality determines the critical mass.  In fact, for the critical mass case,  the equation is exactly this 
gradient flow equation applied to the probability density $\frac{1}{8\pi}\rho$. Hence, in the critical mass case, along solutions 
$\rho(t)$ of the Keller-Segal equation, $\mathcal{H}(\frac{1}{8\pi}\rho(t))$ is decreasing.   It was natural to conjecture that  for such solutions $\rho(t)$,  
the normalized probability densities $\frac{1}{8\pi}\rho(t)$
 approach the manifold $\mathcal{M}$ specified in \eqref{ManDEF} as $t$ approaches infinity, and indeed, the steady sates of the critical mass Keller-Segal 
 equation are exactly the elements of $\mathcal{M}$ multiplied by $8\pi$.  The convergence was proved in \cite{BCC}, 
 and it was shown how the initial data determines the element of 
 $\mathcal{M}$ to which $\frac{1}{8\pi}\rho(t)$ converges, but no estimate on the rate of convergence was given. 

A rate was obtained in \cite{CF} using two stability results. First, a stability result for a sharp Galgliardo-Nirenberg-Sobolev inequality \cite{DD10}
 which plays a central role in \cite{BCC} and \cite{CCL10} was proved and applied to show that for some finite constant $C$ depending on 
 suitable initial data, solutions satisfy
\begin{equation*}
\mathcal{H}(\tfrac{1}{8\pi}\rho(t)) \leq Ct^{-1/8}\ .
\end{equation*}
Second the stability result \eqref{ACLHLSS} result for the log HLS inequality was proved, building on results of \cite{CCL10},  to show that for some other constant $C$, changing in each appearance,
\begin{equation}\label{KSR}
\inf_{g\in \mathcal{M}}\|\rho(t) -8\pi g\|_1  \leq  (\mathcal{H}(\tfrac{1}{8\pi}\rho(t)))^{1/20} \leq Ct^{1/160}\ .
\end{equation}
Using the sharp bound \eqref{mainINEQ} of Theorem~\ref{main}, \eqref{KSR}  improves to $\inf_{g\in \mathcal{M}}\|\rho(t) -8\pi g\|_1 \leq Ct^{-1/16}$.  
With this we close our discussion of the application of Theorem~\ref{main} and turn to its proof. 

\subsection{Components of the proof of Theorem~\ref{main}}

We shall deduce Theorem~\ref{main} by using a duality method introduced in \cite{C16} and applied there to prove a stability 
result for the HLS inequality starting from the stability result for the Sobolev inequality in $d\geq 3$ proved by Bianchi and Egnel \cite{BE91}. 
The inequality corresponding to the Sobolev inequality for $d=2$ is the Onofri inequality \cite{O82}, for which a stability result has been proved very recently  in \cite{CLT23}. However, we shall give another proof which yields a somewhat more explicit result for $d=2$. This proof draws on a theorem of  Gui and Moradifam \cite{GM18} giving an improved form of the Onofri inequality under a certain constraint. I am very grateful to Jean Dolbeault for bringing this paper to my attention, and suggesting that it could be used to prove a stability inequality for Onofri's inequality. 

As we have noted earlier, the log HLS inequality is conformally invariant, and hence by means of the stereographic projection, it can be written in 
an equivalent form as an inequality for functions on $S^2$. The Onofri inequality on $S^2$  is equivalent to the log HLS inequality on 
$S^2$ via a Legendre transform argument as shown in \cite{CL92}. The duality method introduced in  
\cite{C16} may be applied using a stability result for the Onofri inequality, together with a strong form of the Young's inequality relating the relative entropy and its Legendre transform, to prove a stability result for log HLS on $S^2$. This may then be transferred to $\R^2$ by the stereographic projection, yielding Theorem~\ref{main}.

\section{Duality in general}

Partly for the purpose of fixing notation, we briefly describe the duality method for transferring stability results that was introduced in \cite{C16}. 
In this brief recapitulation, we focus on the specific cases we need for the present application, which considerably 
simplifies the general development in \cite{C16}. 

Let $E$ and ${E^*}$ be a dual pair of Banach spaces with respect to the bilinear form $\langle \cdot,\cdot\rangle$ on $E\times {E^*}$. Let $\cE$  be a proper lower semi-continuous  (l.s.c.) convex function on $E$ with values in $(-\infty,+\infty]$ where ``proper'' means that $\cE$ is not identically equal to $+\infty$. The domain of a proper l.s.c. function $\cE$  is the set $D(\cE)$ on which it is finite. 

The Legendre transform  $\cE^*$ of $\cE$ is the function on ${E^*}$ defined by
\begin{equation}\label{ltrandef}
\cE^*(y) = \sup_{x\in E}\{ \langle x,y\rangle - \cE(E)\ \}\ .
\end{equation}
It is evident from the definition that $\cE^*$ is l.s.c. and convex. By the Fenchel-Moreau Theorem, $\cE^*$ is also proper, and the Legendre transform $\cE^{**}$ of $\cE^*$ is $\cE$ itself. That is, the Legendre transform is  an involution on the class proper l.s.c. convex functions. 

The subgradient of $\cE$ at $x_0\in E$ is the subset 
$\partial\cE(x_0)$ of ${E^*}$ consisting of all $y$ such that
\begin{equation}\label{subg}
\cE(x_1) \geq \cE(x_0)  + \langle x_1-x_0,y\rangle \qquad{\rm for\ all}\quad x_1\in E\ .
\end{equation}
That is, $y \in \partial\cE(x_0)$ if and only if the  hyperplane $(x,\cE(x_0) + \langle x-x_0,y\rangle )$ lies below the graph of $x\mapsto \cE(x)$, and intersects it at 
$(x_0,\cE(x_0))$ so that the hyperplane $(x,\cE(x_0) + \langle x-x_0,y\rangle )$ is a supporting plane to the graph of  $x\mapsto \cE(x)$ at $x_0$.

An immediate consequence of the definition \eqref{ltrandef} is that for all $x\in E$ and all $y\in {E^*}$, 
\begin{equation}\label{youngin}
\langle x,y\rangle \leq \cE^*(y) + \cE(x) \ .
\end{equation}
This is known as Young's Inequality. Suppose that for some $x_0\in E$ and $y\in {E^*}$, there is equality in \eqref{youngin}. Since $\langle x_0,y\rangle$ is finite, $x\in D(\cE)$ and $y\in D(\cE^*)$, and for any other $x_1\in E$, 
$$
\langle x_0,y\rangle - \cE^*(y) - \cE(x_0)   \geq  \langle x_1,y\rangle - \cE^*(y) - \cE(x_1)\ ,  
$$
which is equivalent to \eqref{subg}. Thus, necessarily, $y\in \partial\cE(x_0)$.  Conversely, suppose that  $y\in \partial\cE(x_0)$. Then by  \eqref{subg}, for all $x_1\in E$, 
$\langle x_0,y\rangle  \geq \cE(x_0)  + \langle x_1,y\rangle - \cE(x_1)$, and this implies $\langle x_0,y\rangle  \geq \cE(x_0)+\cE^*(y)$. Therefore, equality holds in \eqref{youngin} if and only if $y\in \partial \cE(x)$. By the symmetry in $\cE$ and $\cE^*$, or the same reasoning, it also true that 
equality holds in \eqref{youngin} if and only if $x\in \partial \cE^*(y)$.
In particular, $y\in \partial\cE(x)$ if and only if 
$x\in \partial \cE^*(y)$, and in this sense the set valued functions $\partial \cE$ and $\partial \cE^*$ are inverse to one another. For $A\subset E$,  define $\partial\cE(A) = \cup_{x\in A}\partial\cE(x)$ and likewise  for $B\subset {E^*}$,  define $\partial\cE^*(B) = \cup_{y\in B}\partial\cE^*(y)$.

Consider $x_1,x_2$ in the domain of $\cE$. Let $y\in \partial \cE(x_1)$ so that  $\langle x_1,y\rangle = \cE(x_1) + \cE^*(y)$. By Young's inequality, 
$\langle x_2,y\rangle \leq \cE(x_2) + \cE^*(y)$. Therefore, 
$$
\cE(x_2) \geq \cE(x_1) + \langle x_2-x_1,y\rangle \ ,
$$
which is sometimes called the ``above the tangent line inequality''. 

Now suppose that $\cE$ and $\cF$ are two proper l.s.c. convex functions on $E$. An inequality of the form $\cE(x) \leq \cF(x)$ for all $x\in E$ has an equivalent form expressed in terms of the Legendre transforms $\cE^*$ and $\cF^*$, and it can be advantageous to investigate this pair of inequalities in parallel. 

\begin{thm}\label{dualsharp} $\cE$ and $\cF$ are two proper l.s.c. convex functions on $E$. 
\begin{equation}\label{eqcases2}
  \cE(x)   \leq \cF (x)\quad {\rm for\ all}\ x\in E \quad \iff \quad   \cF^*(y) \leq  \cE ^*(y)
  \quad {\rm for\ all}\ y\in {E^*}\ .
 \end{equation}
  Moreover, let $E_0 := \{\ x\in E\:\ \cE(x)  = \cF (x)\ \}$ and   ${E^*_0} := \{\ y\in {E^*}\:\ \cF^*(y)  = \cE^* (y)\ \}$, which might be empty. In any case
  \begin{equation}\label{eqcases3}
  E_0 = \partial \cE^*({E^*_0})   \quad {\rm and}\quad  {E^*_0} = \partial \cE(E_0)
 \end{equation}
 and
  \begin{equation}\label{eqcases4}
  E_0 = \partial \cF^*({E^*_0})   \quad {\rm and}\quad  {E^*_0} = \partial \cF(E_0)\ .
 \end{equation}
\end{thm}

 \begin{proof}
 Suppose  that the pair of inequalities in \eqref{eqcases2} is valid. Let $E_0 = \{\ x\in E\:\ \cE(x)  = \cF (x)\ \}$ and let  ${E^*_0} = \{\ y\in {E^*}\:\ \cF^*(y)  = \cE^* (y)\ \}$.
 
Fix $x\in E$ and $y\in {E^*}$.  If either  $x\in \partial\cE^*(y)$ $y\in \partial\cE(x)$, then $\langle x,y\rangle = \cE(x) + \cE^*(y)$, and by Young's inequality
 $\langle x,y\rangle \leq \cF(x) + \cF^*(y)$. That is, $\cE(x) + \cE^*(y) \leq \cF(x) + \cF^*(y)$, which is the same as
 \begin{equation*}
 \cE^*(y) - \cF^*(y) \leq \cF(x) -\cE(x)\ .
 \end{equation*}
 It follows that if $\cF(x) =\cE(x)$, then $ \cE^*(y) = \cF^*(y)$.  That is $\partial \cE(E_0) \subseteq {E^*_0}$ and $\partial \cE^*({E^*_0}) \subseteq E_0$.  Since $\partial \cE$ and
 $\partial \cE^*$ are inverse to one another,  $\partial \cE(E_0) = {E^*_0}$ and $\partial \cE^*({E^*_0}) = E_0$, proving \eqref{eqcases3}.
The same reasoning shows that $\partial \cF(E_0) = {E^*_0}$ and $\partial \cF^*({E^*_0}) = E_0$, proving \eqref{eqcases4}.
\end{proof}

Now suppose that  $\cE$ and $\cF$ are two proper l.s.c. convex functions on $E$ such that for all $x\in E$, $\cE(x) \leq \cF(x)$ and such that $E_0 = \{\ x\in E\:\ \cE(x)  = \cF (x)\ \} \neq \emptyset$.  Then an inequality of the form
\begin{equation}\label{quadstab}
 \cE(x) + C \inf_{x_0\in E_0}\| x- x_0 \|_E^2 \leq \cF(x)
\end{equation}
is called a {\em quadratic stability bound for the inequality $\cE \leq \cF$.}

On the basis of Theorem~\ref{dualsharp}, one might hope that such an inequality would always imply, via a simple duality argument, that for some finite constant $C^*$,
\begin{equation}\label{quadstab2}
\cE^*(y) - \cF^*(y) \geq C^* \inf_{y_0\in {E^*_0}}\| y- y_0 \|_{E^*}^2\ .
\end{equation}

However, to deduce an inequality of the form \eqref{quadstab2}, we need in addition to  \eqref{quadstab}, a strong form of the convexity of $\cE^*$.  
To simplify the notation, we discuss this stronger form of convexity and its consequences in terms of $\cE$ instead of $\cE^*$, which are, 
after all, on an equal footing. 

\begin{defi}\label{lamOCNV} Let $\cE$ be a proper l.s.c. convex function on $E$, and let $\lambda > 0$. Then 
$\cE$ is $\lambda$ convex  in case for all $x_0,x\in D(\cE)$, and all $y\in \partial\cE(x_0)$ 
\begin{equation}\label{statl}
\cE(x) \geq  \cE(x_0) + \langle x-x_0,y\rangle + \lambda \|x-x_0\|_E^2\ .
\end{equation}
\end{defi}

The part of the next lemma that is crucial for us is the $\lambda$ convexity implies a strong form of Young's inequality, \eqref{stryou}. 

\begin{lem}\label{SYL}  Let $\cE$ be a proper l.s.c. convex function on $E$. Suppose that for some $\lambda >0$ $\cE$ is $\lambda$ convex. Then for all $y$ with $\partial \cE^*(y) \neq \emptyset$,  $\partial \cE^*(y)$ is a singleton, so that $\cE^*$ has a gradient at $y$,   $\nabla \cE^*(y)$.
Moreover, the gradient is Lipschitz with constant $(2\lambda)^{-1}$:
\begin{equation}\label{lipbnd}
\|\nabla \cE^*(y_1) - \nabla\cE^*(y_0) \|_E \leq \frac{1}{2\lambda} \| y_1-y_0 \|_{E^*}\ .
\end{equation}

Finally we have the following strong form of Young's inequality:
\begin{equation}\label{stryou}
\cE(x) + \cE^*(y) \geq \langle x,y\rangle + \lambda \| x- \nabla\cE^*(y) \|_E^2\ .
\end{equation}
\end{lem}

\begin{proof} Let $x\in E$ and $y\in {E^*}$, and let $x_0\in \partial \cE^*(y)$ so that $y\in \partial\cE(x_0)$. Then by the $\lambda$ convexity of $\cE$,
\begin{eqnarray*}
\cE(x) &\geq&  \cE(x_0) + \langle x-x_0,y\rangle + \lambda \|x-x_0 \|_E^2\nonumber\\
&=& \langle x,y\rangle  - \left(\langle x_0,y\rangle -  \cE(x_0)\right)  + \lambda \|x-x_0 \|_E^2\nonumber\\
&\geq&  \langle x,y\rangle - \cE^*(y) + \lambda \|x-x_0 \|_E^2 \label{stryoun2}
\end{eqnarray*}
That is,
\begin{equation}\label{stryoun3}
\cE(x) + \cE^*(y) \geq \langle x,y\rangle + \lambda \|x-x_0 \|_E^2\ .
\end{equation}
Now suppose that $x\in \partial \cE^*(y)$.  Then $\cE(x) + \cE^*(y) = \langle x,y\rangle$, and hence $\|x-x_0 \|_E = 0$, so that $x=x_0$. This shows that 
$\partial\cE^*(y) = \{x_0\} =\{\nabla\cE^*(y)\}$. Replacing $x_0$ in \eqref{stryoun3} with $\nabla \cE^*(y)$ yields \eqref{stryou}. 

Finally, for any $y_0,y_1\in {E^*}$, define $x_j := \nabla \cE^*(y_j)$, $j=0,1$.  Then from \eqref{statl},
\begin{eqnarray*}
\cE(x_1) -  \cE(x_0) &\geq& \langle x_1-x_0,y_0\rangle + \lambda \|x_1-x_0 \|_E^2\\
\cE(x_0)  - \cE(x_1) &\geq& \langle x_0-x_1,y_1\rangle + \lambda \|x_1-x_0 \|_E^2\\
\end{eqnarray*}
Summing both sides  yields
$$
 2 \lambda \|x_1-x_0 \|_E^2 \leq \langle x_1-x_0,y_1 -y_0\rangle \leq \|x_1-x_0 \|_E \|y_1-y_0 \|_{E^*}\ .
$$
This proves \eqref{lipbnd} since  $x_j := \nabla \cE^*(y_j)$, $j=0,1$.
\end{proof}

The next theorem provides the transfer of stability that we seek. 

\begin{thm}\label{DUSTAB} Let $\cE$ and $\cF$ be proper l.s.c. convex functions on $E$ such that $\cE(x) \leq \cF(x)$ for all $x\in E$, and $E_0 = \{\ x\ :\ \cE(x) = \cF(x)\ \} \neq \emptyset$. 
Suppose further that the inequality for some constant $C$,  $\cE(x) \leq \cF(x)$ is quadratically  stable for some norm $\|\cdot \|_E$, and the function $\cE^*$ is $\lambda$ convex with resect to some norm $\|\cdot \|_{E^*}$. Then for all $y\in {E^*}$,
\begin{equation}\label{dualstab}
\cE^*(y) - \cF^*(y)  \geq   \frac{\lambda\mu}{2} \inf_{y_0\in {E^*_0}} \|y- y_0 \|_{E^*}^2\ .
\end{equation}
\end{thm}

\begin{proof}  Since $\cE^*$ is $\lambda$ convex, the strong Young's inequality yields that for all $x\in E$ and $y\in Y$,
\begin{equation}\label{pr1}
\cE(x) + \cE^*(y) \geq \langle x,y\rangle + \lambda \|y- \nabla\cE(x) \|_{E^*}^2\ .
\end{equation}
We also have the stability inequality
\begin{equation}\label{pr2}
\cF(x) - \cE(x) \geq  C \inf_{x_0\in E_0}\| x- x_0 \|_E^2\ .
\end{equation}
Summing both sides of \eqref{pr1} and  \eqref{pr2}, rearranging terms, and using the fact that  $\cF^*(y) \geq \langle x,y\rangle - \cF(x)$ for all $x$, yields
\begin{equation*}
\cE^*(y) - \cF^*(y)  \geq  C \inf_{x_0\in E_0}\| x- x_0 \|_E^2 + \lambda \|y- \nabla\cE(x) \|_{E^*}^2\ .
\end{equation*}

Fix $\epsilon>0$, and choose $\hat{x}_0\in E_0$ such that $\| x- \hat{x}_0 \|_E^2 \leq \inf_{x_0\in E_0}\| x- x_0 \|_E^2+ \epsilon$. Then since
$\| x- \hat{x}_0 \|_E \geq 2\lambda \| \nabla \cE(x) - \nabla \cE(\hat{x}_0) \|_{E^*}$,
$$
\cE^*(y) - \cF^*(y)  \geq \lambda\left(4C\lambda  \| \nabla \cE(x) - \nabla \cE(\hat{x}_0) \|_{E^*}^2 +    \|y- \nabla\cE(x) \|_{E^*}^2 \right) - C\epsilon
$$
Defining $\mu := \min\{4C\lambda\ ,\ 1\}$, and using the elementary inequality $2(a^2+b^2) \geq (a+b)^2$, followed by the triangle inequality.
\begin{eqnarray*}
\cE^*(y) - \cF^*(y)  &\geq& \frac{\lambda\mu}{2}\left( \| \nabla \cE(x) - \nabla \cE(\hat{x}_0) \|_{E^*} +    \|y- \nabla\cE(x) \|_{E^*} \right)^2- C\epsilon\\
 &\geq& \frac{\lambda\mu}{2}  \|y- \nabla\cE(\hat{x}_0) \|_{E^*}^2- C\epsilon\ .
\end{eqnarray*}
Then since $\nabla\cE(\hat{x}_0) \in {E^*_0}$ and $\epsilon$ is arbitrary, \eqref{dualstab} is proved.
\end{proof}

\begin{rem} A somewhat more general version of Theorem~\ref{DUSTAB} was applied in \cite{C16} to prove stability for the 
HLS inequality using the quadratic stability bound for the Sobolev inequality of Bianchi and Egnel \cite{BE91} together with the 
fact that the function $f\mapsto \|f\|_p^2$ is $(p-1)$-convex for $1\leq p \leq 2$, which is proved in  \cite{C16}.   
The Bianchi-Egnel result  has since been improved by Dolbeault, Esteban, Figalli, Frank and Loss \cite{DEFFL}. 
To write their result in the language used here,  for $d\geq 3$, let $E = L^{2d/(d-2)}(\R^d)$ and $E^* = L^{2d/(d+2)}(\R^d)$ with
$\langle f,g\rangle = 2\int_{\R^d} fg\d x$. Let $\cE(f) = \|f\|_E^2$ so that $\cE^*(g) = \|g\|_{E^*}^2$. 
Let $\cF(f) = S_d\|(-\Delta)^{1/2}f\|_2^2$, where $S_d$ is the sharp Sobolev constant,  so that $\cF^*(g) = S_d^{-1}\|(-\Delta)^{-1/2}g\|_2^2$. 
Then the result of  \cite{DEFFL} is that there exists a computable constant $\beta$, independent of $d$, so that
\begin{equation}\label{DEFFL1}
\cF(f) - \cE(f) \geq \frac{\beta}{d} \inf_{\tilde{f}\in E_0}\| \nabla (f - \tilde{f})\|_2^2
\end{equation}
where $E_0$ is the manifold of optimizers of the Sobolev inequality.  By the Sobolev inequality itself, this implies the weaker bound
\begin{equation}\label{DEFFL2}
\cF(f) - \cE(f) \geq \frac{\beta}{dS_d} \inf_{\tilde{f}\in E_0}\| f - \tilde{f}\|_E^2\ .
\end{equation}
Then since $1 \leq 2d/(d+2) \leq 2$ for all $d>3$, 
if follows from the $(p-1)$-convexity of $f\mapsto \|f\|_p^2$ that $\cE^*$ is $\frac{d-2}{d+2}$ convex. With this and  \eqref{DEFFL2}, 
we have all that is needed to apply
Theorem~\ref{DUSTAB} as in \cite{C16}, and the result is that
\begin{equation}\label{HLSstabold}
\cE^*(g) - \cF^*(g) \geq \min\left\{  \frac{\beta (d-2)^2}{d+2}|S^d|^{d/2}\ , \ 1\right\} \frac{d-2}{d+2}  \left( \inf_{\tilde{g}\in E^*_0}\|g-\tilde{g}\|_{2d/(d+2)}^2 \right)
\end{equation}
where $|S^d|$ is the volume of the $d$-dimensional unit sphere in $\R^{d+1}$, which figures in the sharp Sobolev constant $S_d$, and $E^*_0$ is the set of HLS optimizers. 

Recall that the norm $\|\cdot\|_{E^*}$ is the $L^{2d/(d+2)}(\R^d)$ norm, the strongest norm figuring in the HLS inequality.  At least as far as the norm
and the quadratic power are concerned, The bound in \eqref{HLSstabold} is optimal.  It is remarkable that this was obtained using 
only the weaker stability bound \eqref{DEFFL2}, and not the full strength of \eqref{DEFFL1}. However, near the manifold of Sobolev optimizers, 
there is in some sense little difference between \eqref{DEFFL1} and \eqref{DEFFL2}.

A second observation is that since  $f\mapsto \|f\|_p^2$ is {\em not}  $\lambda$-convex for any $\lambda>0$ and $p>2$, and since $2d/(d-2) > 2$ for all $d>3$, 
Theorem~\ref{DUSTAB} cannot be used to derive a version of \eqref{DEFFL2} with any constant from a 
version of \eqref{HLSstabold} with any constant, let alone a version of \eqref{DEFFL1}.  The results in \cite{C16} would provide an analog of \eqref{DEFFL2} with an exponent higher than $2$. It would be interesting to have a more symmetric understanding of how stability relates inequalities such as  \eqref{DEFFL2} and \eqref{HLSstabold}, or better yet,    \eqref{DEFFL1} and \eqref{HLSstabold}.
\end{rem}

\section{Entropy inequalities}

The main result of this section is the strong Young's inequality for the entropy, which turns out to be a consequence of  the optimal 
$2$-uniform convexity inequality for $L^p$ spaces (see \cite{BCL}). We also recall some  entropy inequalities commonly used in the theory of large deviations \cite[Chapter 10]{V84} that we shall use here. 

Let $(X,\mathcal{B},\nu)$ be  a measure space. A probability density $\rho$ on $X$ is a non-negative 
measurable  function on $X$ such that $\int_X \rho \d \nu =1$. Let $\rho_1,\rho_2$ be two probability densities on $X$. Then the relative entropy 
$H(\rho_1|\rho_2)$ of $\rho_1$ with respect to $\rho_2$ is defined by
$$
H(\rho_1|\rho_2) = \begin{cases} \int_X\rho_1(\log \rho_1 - \log \rho_0)\d \nu  &  \int_{\{\rho_0 = 0\}}\rho_1 \d \nu =0 \\
\phantom{X}\quad\quad +\infty & \phantom{X} {\rm otherwise} \end{cases}\ .
$$
When $ \int_{\{\rho_0 = 0\}}\rho_1 \d \nu =0$, then
$$
H(\rho_1|\rho_2) = \int_X \left(\frac{\rho_1}{\rho_0}\right)\log \left(\frac{\rho_1}{\rho_0}\right)\rho_0 \d \nu \geq 0
$$
by Jensen's inequality, and there is equality if and only if $\rho_1 = \rho_0$. In fact, there is a stability result for this inequality, namely Pinsker's inequality \cite{P64}:
\begin{equation}\label{Pinsker}
H(\rho_1|\rho_0)  \geq \frac12 \left(\int_X|\rho_1 - \rho_0|\d \nu\right)^2\ .
\end{equation}
This is  Pinsker's inequality with the sharp constant,  due to Csisz\'ar and Kullback \cite{Cs67,K67}. 
Although the lower bound \eqref{Pinsker} on $H(\rho_1|\rho_0)$ is symmetric in $\rho_0$ and $\rho_1$, in general $H(\rho_1|\rho_0)$ and 
$H(\rho_0|\rho_1)$  are incomparable -- even on a finite probability space, either may be infinite while the other is finite, depending on the supports of the two densities. 

By definition, if $H(\rho_1|\rho_0)$ is finite and $A\subset X$ satisfies $\int_A\rho_0\d \nu =0$, then $\int_A\rho_1 \d \nu =0$. In fact, more is true. If 
$H(\rho_1|\rho_0)$ is finite and 
$\int_A\rho_0\d \nu$ is small then so is $\int_A\rho_1\d \nu$. A version of this may be found in \cite{V84}, for example. Here is a simple version suited to our purposes.

\begin{lem}\label{LDL} Let $\rho_0,\rho_1$ be two probability densities on $(X,\mathcal{B},\nu)$ such that  $H(\rho_1|\rho_0) < \infty$.  Then for any $\epsilon>0$, there is a $\delta>0$ depending only 
$\epsilon$ and $H(\rho_1|\rho_0)$ such that if $\int_A\rho_0\d \nu \leq \delta$, then   $\int_A\rho_1\d \nu \leq \epsilon$
\end{lem}

\begin{proof} We use the elementary fact that for all $s\geq 0$ and $t\in \R$, $st \leq s\log s + e^{t-1}$. Let $a<0$ and define $f := a1_A $ where $1_A$ denotes the indicator function of $A$. Let $g:= \rho_1/\rho_0$
$$
a\int_A \rho_1\d \nu = \int_X f g \rho_0\d \nu \leq \int_X f\log f \rho_0 + \left( e^a\int_A\rho_0\d \nu + \int_{X\backslash A}\rho_0\d \nu \right)e^{-1}
$$
Therefore
$$
\int_A \rho_1\d \nu \leq \frac1a H(\rho_1|\rho_0) + \frac{e^a}{a} \int_A \rho_0\d \nu\ .
$$
Now choose $a = 2H(\rho_1|\rho_0)/\epsilon$, and then $\delta = \frac{a\epsilon}{2}e^{-a}$ for this $a$. 
\end{proof}

In the rest of this section, suppose that $\nu(X) =1$ so that the constant function $1$ is a probability density. Our next results concern the function $\rho \mapsto H(\rho|1) = \int_X \rho \log \rho \d \nu$ on the domain consisting of probability densities on 
$(X,{\mathcal B},\nu)$, so that it is $+\infty$ in the rest of $L^1(X,{\mathcal B},\nu)$.  To simplify our notation, we write $H(\rho)$ in place of $H(\rho|1)$.
It is well known (see for instance \cite{DV83,V84}) and also easy to check that the Legendre transform $H^*$  of $H$ is the function on  $L^\infty(X,{\mathcal B},\nu)$  given by
\begin{equation*}
H^*(\varphi) = \sup_{\rho\in L^1}\left\{ \int_X \varphi \rho  \d \nu - H(\rho)\ \right\} = \log\left(\int_X e^\varphi \d \nu\right)\ ,
\end{equation*}
and since both $H$ and $H^*$ are proper l.s.c. convex functions,
\begin{equation}\label{DV2}
H(\rho) = \sup_{\rho\in L^1}\left\{ \int_X \varphi \rho  \d \nu - \log\left(\int_X e^\varphi \d \nu\right)\ \right\} \ .
\end{equation}
Consequently, we have the  Young's inequality
$$
H(\rho) + H^*(\varphi) \geq  \int_X \varphi \rho  \d \nu\ .
$$
Note that $\nabla H^*(\varphi) =\left(\int_X e^\varphi \d \nu\right)^{-1} e^\varphi$, which is evidently a probability density. 
The next theorem gives the crucial strong Young's inequality that we shall use. 

\begin{thm}\label{STRYO} For all probability densities $\rho$ and all  $\varphi \in L^\infty(X,{\mathcal B},\nu)$
\begin{equation}\label{SRTYO1}
H(\rho) + H^*(\varphi) \geq  \int_X \varphi \rho  \d \nu + \frac12 \| \rho - \nabla H^*(\varphi)\|_1^2\ .
\end{equation}

\end{thm}

\begin{proof} We first claim that the function $\rho \mapsto H(\rho)$ is $\tfrac12$-convex on the set of probability densities on $(X,{\mathcal B},\nu)$. That is, for any two probability densities $\rho_1,\rho_0$, 
\begin{equation}\label{SCONRE}
H(\rho_1) - H(\rho_0) \geq \int_X( \rho_1 - \rho_0) \log \rho_0 \d \nu + \frac12 \| \rho_1 - \rho_0\|_1^2 \ .
\end{equation}

This is a limiting case of a result proved in \cite{C16}. In \cite[Theorem 2.5]{C16}, it is proved, using the sharp $2$-uniform convexity of $L^p$ spaces, $p\in (1,2]$, that for all such $p$, 
\begin{equation*}
\|\rho_1\|_p^2 -  \|\rho_0\|_p^2  \geq 2 \int_X( \rho_1-\rho_0)\|\rho_0\|_p^{2-p}\rho_0^{p-1}\d \nu +  (p-1)\|\rho_1-\rho_0\|_p^2\ 
\end{equation*} 
This can be rewritten as
\begin{equation*}
\frac{\|\rho_1\|_p^2 - 1}{p-1} -  \frac{\|\rho_0\|_p^2-1}{p-1} \geq  2\|\rho_1\|_p^{2-p} \int_X( \rho_1-\rho_0)
\left(\frac{\rho_0^{p-1} -1}{p-1} \right)\d \nu + \|\rho_1-\rho_0\|_p^2\ 
\end{equation*}
Taking the limit $p\downarrow 1$ yields \eqref{SCONRE}.   
Note that $\log \rho_0 = \nabla H(\rho |1)$, so this \eqref{SCONRE} is a strong Young's inequality in the sense of Definition~\ref{lamOCNV}. Now
\eqref{SRTYO1} follows from Lemma~\ref{SYL}.
\end{proof}

\section{Stability for Onofri's inequality on $S^2$}

Onofri's inequality \cite{O82} says that 
\begin{equation}\label{onofineq}
\frac{1}{4}\int_{S^2}|\nabla u|^2 \d \sigma \geq \log\left(\int_{S^2}e^u\d \sigma\right) -\int_{S^2} u \d \sigma \ .
\end{equation}
where $\sigma$ denotes the uniform probability measure on $S^2$.

Notice that  adding a constant to $u$ does not change either side of the inequality.  This invariance may be used to 
fix the value of $\int_{S^2}e^u\d \sigma$. In fact, \eqref{onofineq} arose in an investigation by Onofri and Virasoro \cite{OV82} of a random surface model of Polykov \cite{P81}, and one of their key insights was that the theory was better-behaved under a constraint on the total area of the random surface, and this amounts to fixing the value of  $\int_{S^2}e^u\d \sigma$.  In many ways, the natural variable is $\rho := e^u$, and several formulas in \cite{OV82,P81} are written this way. 

Therefore,  we suppose that 
\begin{equation}\label{onofineqM}
\int_{S^2}e^u\d \sigma = 1 \ ,
\end{equation}
so that we may think of $\rho := e^u$ as a probability density on $S^2$.  Let $\rho_0$ denote the uniform probability density $\rho_0(y) = 1$ for all $y\in S^2$.
There are two relative entropies associated to the pair of probability densities $\rho$ and $\rho_0$, namely
\begin{equation}\label{onofineqE1}
H(\rho || \rho_0) = \int_{S^2}\rho(\log \rho - \log \rho_0)\d \sigma = \int_{S^2} e^u u \d \sigma
\end{equation}
and
\begin{equation}\label{onofineqE2}
H(\rho_0 || \rho) = \int_{S^2}\rho_0(\log \rho_0 - \log \rho)\d \sigma = - \int_{S^2}  u \d \sigma\ .
\end{equation}
Thus, under the normalization \eqref{onofineqM} and written in terms of the variable $\rho = e^u$, the Onofri inequality can be written as
\begin{equation}\label{onofineqEV}
\int_{S^2}|\nabla \log \rho|^2 \d \sigma \geq  4 H(\rho_0 || \rho) \ .
\end{equation}
Written in this way, the Onofri inequality expresses the relative entropy  dissipation of the heat equation on the sphere for 
the ``reversed'' relative entropy in \eqref{onofineqE2}. Indeed, define $\rho(t) = e^{t\Delta}\rho$, where $\Delta$ is the Laplacian. Then for any initial probability density $\rho$, at any time $t>0$,  $\rho(t)$ is smooth and uniformly bounded below. Simple computations yield
\begin{equation}\label{onofineqEP}
\frac{{\rm d}}{{\rm d}t}H(\rho_0 || \rho) = -\int_{S^2} \frac{1}{\rho(t)}\Delta \rho(t) \d \sigma = - \int_{S^2}|\nabla \log \rho(t)|^2 \d \sigma\ .
\end{equation}
Hence for solutions $\rho(t)$ of the heat equation on $S^2$ with strictly positive initial data $\rho_0$, 
\begin{equation}\label{onofineqEP2}
H(\rho(t) || \rho)  \leq e^{-4t}H(\rho_0 || \rho)  \ .
\end{equation}

\subsection{Conformal invariance of Onofri's inequality} Let $\tau: S^2\to S^2$ be a conformal transformation. Then it is well known that for $u\in H^1(S^2)$, 
$$
\int_{S^2}|\nabla u\circ \tau |^2\d \sigma = \int_{S^2}|\nabla u |^2\d \sigma\ .
$$
The Onofri inequality is conformally invariant, but under a different action of the conformal group on  $u\in H^1(S^2)$, namely
\begin{equation}\label{onofineqCINV}
u \mapsto u\circ \tau + \log \mathcal{J}_\tau := U_\tau\ ,
\end{equation}
where $\mathcal{J}_\tau$ is the Jacobian of the conformal transformation $\tau$.  Under this action of the conformal group, $\int_{S^2}|\nabla u|^2\d \sigma$ is no invariant, but $\int_{S^2}e^u \d\sigma$ is. With $U_\tau$ given by \eqref{onofineqCINV},
\begin{equation*}
\int_{S^2}e^{U_\tau}\d \sigma = \int_{S^2}e^{u\circ \tau}\mathcal{J}_\tau d\sigma = \int_{S^2}e^u\d \sigma \ . 
\end{equation*}
Thus the normalization condition \eqref{onofineqM} under which the Onofri inequality reduces to \eqref{onofineqEV} is conformally 
invariant under this action of the conformal group, which henceforth is the only one that we shall refer to in connection with the Onofri inequality.

To display the conformal invariance, and to make use of it, the stereographic projection turns out to be very convenient. We use the stereographic projection
$\mathcal{S}: S^2\to \R^2$ based on the ``south pole'' $(0,0,-1)\in S^2$, and we recall some relevant formulas. We denote points in $S^2$ by $\omega = (\omega_1,\omega_2,\omega_3)$ Then for $\omega\in S^2$,  $\mathcal{S}(\omega)$ is the point at which the line through $(0,0,-1)$ and $\omega$ intersects the plane $\omega_3 =0$. That is,
\begin{equation*}
\mathcal{S}(\omega_1,\omega_2,\omega_3) = \begin{cases} \displaystyle{\frac{(\omega_1,\omega_2)}{1+\omega_3}} & \omega_3 > -1\\
\phantom{xx}\infty & \omega_3 = -1\end{cases}\ .
\end{equation*}
We use coordinates $x = (x_1,x_2)$ in $\R^2$. Then the inverse transformation is given by
\begin{equation*}
\mathcal{S}^{-1}(x_1,x_2) = (1+|x|^2)^{-1}(x_1,x_2, 1- |x|^2))\ .
\end{equation*}
It follows that ${\displaystyle \omega_3 = \frac{1-|x|^2}{1+|x|^2}}$,  from which one has the useful formula
\begin{equation}\label{stereonorm}
|x|^2 = \frac{1+\omega_3}{1-\omega_3}\ .
\end{equation}

The Jacobian of the stereographic projection is  ${\displaystyle \frac{4}{(1+|x|^2)^2}}$ image of $\sigma$  under the stereographic projection is 
$$
\d \mu := \frac{1}{\pi}\frac{1}{(1+|x|^2)^2}\d x\ .
$$
As in \cite{CL88}, we now write the Onofri inequality as an inequality for functions on $\R^2$ 
using the stereographic projection in order to display its conformal invariance. 
Define the function $\varphi$ on $\R^2\cup\{\infty\}$ by 
\begin{equation*}
\varphi(x) = \begin{cases} -2\ln(1+|x|^2) -\log \pi &x \neq \infty\\
\phantom{xxx}0 & x= \infty
\end{cases}\ .
\end{equation*}
Then
\begin{equation}\label{IBPF}
\d \mu = -\frac{1}{8\pi}\Delta \varphi(x)\d x = e^{\varphi(x)}\d x\ .
\end{equation}
Given a function $u$ on $S^2$ that is continuous at $\omega_0$, define the function $f$ on $\R^2$ by 
\begin{equation}\label{furel}
f= u\circ \mathcal{S}^{-1} + \varphi\ .
\end{equation}
Simple computations using \eqref{IBPF} show that with $\omega_0 := (0,0,-1)$ and $u$ continuous near $\omega_0$,
\begin{equation}\label{COINVA}
\frac{1}{16\pi} \int_{\R^2}(|\nabla |^2 - |\nabla \varphi|^2)\d x = \frac14\int_{S^2}|\nabla u|^2\d \sigma + \int_{S^2}u\d \sigma + u(\omega_0)\ ,
\end{equation}
and that
\begin{equation}\label{COINVB}
\int_{\R^2}e^{f} \d x =  \int_{S^2}e^u\d \sigma \ .
\end{equation}

The left sides of \eqref{COINVA} and \eqref{COINVB} are invariant under $f(x) \mapsto e^{2t}f(e^tx = x_0)$, $t\in \R$, $x_0\in \R$; i.e., 
under dilations and translations, and these transformation do not move the point at infinity.   Hence the Onofri inequality is invariant under these transformations, and it is evidently invariant under rotations on $S^2$. Altogether, we have invariance under the full $6$-dimensional conformal group of $S^2$.

It follows that  if $u = \log(\mathcal{J}_\tau)$, then there is equality in \eqref{onofineq}. This yields a $3$-parameter family of optimizers of the Onofri inequality since when $\tau$ is a rotation on $S^2$, $\mathcal{J}_\tau =1$. We get a simple expression for this family by first working out the Jacobians of the dilations, and then rotating those using non-axial rotations. 

Take the constant optimizer $u = 0$ on $S^2$, and  and let $\tau$ correspond to the Euclidean dilations $x\mapsto 4e^t x$, $t\in\R$. To compute the conformal transform $U_\tau$, 
consider the image $f$ of $u$ under the stereographic projection as given in \eqref{furel}, so that
$f = -2\log(1+|x|^2) - \log \pi $. Applying $\tau$, this becomes $-2 \log (1+ e^{2t}|x|^2) \log \pi - 2t$, and using \eqref{furel} once more,
$$
U_\tau\circ \mathcal{S}^{-1} = 2t - 2\log(1+e^{2t}|x|^2) + 2\log(1+|x|^2)\ .
$$
Now using \eqref{stereonorm}, we have
$U_\tau(\omega) = -2\log(\cosh t + \sinh t \omega_3)$.  Now using a rotation that moves the north pole $(0,0,1)$ to any other point ${\bf n}\in S^2$, we 
get the set of  optimizers
\begin{equation}\label{Mspec}
\mathcal{M}_O := \{ \ u_{t,{\bf n}}(\omega) =   -2\log(\cosh t + \sinh t {\bf n}\cdot \omega)\ ,\quad t\geq 0\ , {\bf n}\in S^2\ \}\ .
\end{equation}
(We have restricted to $t\geq 0$ since for $t<0$, $u_{t,{\bf n}} = u_{-t,-{\bf n}}$).
By construction, for each $t$ and ${\bf n}$, $\int_{S^2}e^{u_{t,{\bf n}}}\d \sigma = 1$. Onofri 
proved that every optimizer that  satisfies the normalization condition \eqref{onofineqM} is of this form. That is, ${\mathcal M}_O$ is the 
manifold of all normalized optimizers of the Onofri inequality. Another proof of this, using competing symmetries \cite{CL90}, may be found in \cite{CL88}.

\begin{lem}\label{closest} Let $u$ be measurable function on $S^2$ such that $\int_{S^2}e^{u}\d \sigma = 1$ and $\int_{S^2}|\nabla u|^2\d \sigma < \infty$.
Then there exists $v_0\in {\mathcal M}_O$ such that $H(e^u|e^{v_0}) \leq H(e^u|e^{v})$ for all $v\in \mathcal{M}_O$. 
\end{lem} 

\begin{proof} By the Moser-Trudinger inequality \cite{M71}, for some $\alpha>0$, $\int_{S^2}e^{\alpha u^2}\d \sigma < \infty$, and hence 
\begin{equation}\label{ONOPT1}
(t,{\bf n}) \mapsto \int_{S^2} e^u (u -u_{t,{\bf n}})\d \sigma = H(e^u|e^{u_{t,{\bf n}}})
\end{equation}
 is a continuous real-valued function on $[1,\infty)\times S^2$. 
For any fixed ${\bf n}$, as $t$ tends to $\infty$, $u_{t,{\bf n}}$ converges uniformly to 
$-\infty$ in the complement of any neighborhood of 
$\{{\bf n}\}$. More precisely, define $A_{\epsilon,{\bf n}} := \{\ \omega :\ \omega\cdot{\bf n} \geq 1-\epsilon \ \}$.   Since $H(e^u|1)$ is finite, Lemma~\ref{LDL} gives us a bound, independent of ${\bf n}$ by symmetry, on how small $\epsilon$ must be to ensure that 
$\int_{ A_{\epsilon,{\bf n}}}e^u \d \sigma < 1/2$.
Choose such an $\epsilon >0$. Then  
\begin{equation}\label{ONOPT}
-\int_{S^2} e^u u_{t,{\bf n}}\d \sigma \geq   \int_{S^2} e^u (u_{t,{\bf n}})_-  \d \sigma \geq \int_{A_{\epsilon,{\bf n}}} e^u (u_{t,{\bf n}})_-\d \sigma\ .
\end{equation}
For any $\lambda>0$, there is a $t_\epsilon$ so that $u_{t,{\bf n}} < -\lambda$ on $A_{\epsilon,{\bf n}}$ for all  $t>t_\epsilon$, 
and again by symmetry, 
$t_\epsilon$ does not depend on ${\bf n}$.  It then follows from \eqref{ONOPT} that for all $t>t_\epsilon$, $-\int_{S^2} e^u u_{t,{\bf n}}\d \sigma \geq \lambda/2$.  Choosing $\lambda = 3H(e^u|1)$, it follows that the minimum of  continuous function \eqref{ONOPT1} 
over the compact set $[0,t_\epsilon]\times S^2$ is is a global minimum.  
\end{proof}

Lemma~\ref{closest} has a simple corollary that is a result of Onofri \cite{O82}, whose argument was topological. In fact, this is a key point in Onofri's proof; 
it allowed him to apply then recent work of Aubin \cite{A79} on the functional $J(u) := \frac14 \int_{S^2}|\nabla u|^2\d \sigma 
 -\log\left(\int_{S^2}e^u\d \sigma\right)  + \int_{S^2}u \d \sigma$. 

In the present context, it is desirable to relate Corollary~\ref{ONCR} below to a ``distance'' minimization problem, as in the proof given here. 

\begin{cor}\label{ONCR} Let $u$ be measurable function on $S^2$ such that $\int_{S^2}e^{u}\d \sigma = 1$ and $\int_{S^2}|\nabla u|^2\d \sigma < \infty$.
 Then there exists a conformal transformation $\tau$ such that with $U_\tau = u\circ \tau + \log(J_\tau)$,  
\begin{equation}\label{ONOPT2}
\int_{S^2}e^{U_\tau} \omega \cdot {\bf n}\d \sigma = 0
\end{equation}
for all ${\bf n}\in S^2$. 
\end{cor} 

\begin{proof} Let $v_0\in \mathcal{M}_O$ be chosen to minimize $H(e^u|e^v)$, noting that the minimizer exists by Lemma~\ref{closest}. Now let $\tau$ be the conformal transformation under which the image of $v_0$ is $1$. Then $H(e^{U_\tau}|1) = H(e^u|e^{v_0})$ and 
$H(e^{U_\tau}|1) \leq H(e^{U_\tau}|u_{t,{\bf n}})$ for all $t$ and ${\bf n}$. Now differentiating in $t$ at $t=0$ yields \eqref{ONOPT2}.
\end{proof}

\begin{rem} Let $\rho$ be a probability density on $S^2$. The condition that $\int_{S^2} \rho  \omega \cdot {\bf n}\d \sigma = 0$ for all $\omega$ 
means that $\rho$ cannot be too concentrated around any $\omega \in S^2$, since if $\rho$ were concentrated around some $\omega$, 
the integral would have a value close to $1$ for that $\omega$.  Corollary~\ref{ONCR} says that for any $\rho$, there is a conformal image of $\rho$ that is suitably far from being concentrated at a point.  For this reason, concentration-compactness arguments do not enter Onofri's proof.
\end{rem}

Let $u$ and $v$ be as in Lemma~\ref{closest}. Let $\tau^{-1}$ be the conformal transformation such that $v= \log(\mathcal {J}_\tau)$.  
Let $U_\tau$ be the image of $u$ under the conformal transformation $\tau$. Evidently the conformal image of $v$ under $\tau$ is the constant 
optimizer $0$.   By the conformal invariance of the relative entropy,
$H(e^{U_\tau}|1) \leq H(e^{U_\tau}|e^{u_{t,{\bf n}}})$, or, what is the same
$$
0 \geq \int_{S^2} U_\tau e^{u_{t,{\bf n}}}v\d \sigma
$$
for all $t$ and ${\bf n}$, with equality at $t=0$. Making a Taylor expansion in $t$, we conclude
$$
\int_{S^2} U_\tau  (\omega\cdot {\bf n})\d \sigma = 0
$$
for all ${\bf n}$. 

A strong stability inequality for Onofri's inequality can be deduced from a theorem of Gui and Moradifam \cite{GM18} that  proved a conjecture of  
Chang and Yang \cite{CY88} that has its roots in work of Aubin \cite{A79} and Moser \cite{M71}.  See \cite{GM18} for more on the history of the solution of this conjecture, to which many people contributed.

\begin{thm}\label{ONSTX} Let $u$ be measurable function on $S^2$ such that $\int_{S^2}e^{u}\d \sigma = 1$ and 
$\int_{S^2}|\nabla u|^2\d \sigma < \infty$. Then
\begin{equation}\label{ONSTX1}
\frac{1}{4}\int_{S^2}|\nabla u|^2 \d \sigma - \ln\left(\int_{S^2}e^u\d \sigma\right) + \int_{S^2} u \d \sigma \geq 
\frac18 \inf_{v\in \mathcal{M}_O} \left\{ \int_{S^2}|\nabla u -  \nabla v|^2\d \sigma\ \right\}\ .
\end{equation}
and
\begin{equation}\label{ONSTX1B}
\frac{1}{4}\int_{S^2}|\nabla u|^2 \d \sigma - \ln\left(\int_{S^2}e^u\d \sigma\right) + \int_{S^2} u \d \sigma \geq 
\frac12 \inf_{v\in \mathcal{M}_O} \left\{\ H(e^v|e^u)\ \right\}\ .
\end{equation}

 \end{thm}

\begin{proof} The theorem of Gui and Moradifam \cite{GM18} states that for
 all $u$ satisfying 
\begin{equation}\label{onofineqRTC}
\int_{S^2} e^{u(\omega)} \omega\cdot {\bf n}\d \sigma = 0\quad{\rm for\ all}\  {\bf n}\in S^2, 
 \end{equation}
\begin{equation}\label{onofineqRT}
\frac{1}{8}\int_{S^2}|\nabla u|^2 \d \sigma - \ln\left(\int_{S^2}e^u\d \sigma\right) + \int_{S^2} u \d \sigma \geq 0 \ .
\end{equation}
(The initial result in this direction was Aubin's proof \cite{A79} that under this constraint, the functional in \eqref{onofineqRT} is bounded below. 
This plays a key role in Onofri's proof of his inequality.)
That is, for all such $u$,
\begin{equation}\label{onofineqRT2}
\frac{1}{4}\int_{S^2}|\nabla u|^2 \d \sigma - \ln\left(\int_{S^2}e^u\d \sigma\right) + \int_{S^2} u \d \sigma \geq \frac{1}{8}\int_{S^2}|\nabla u|^2 \d \sigma
\end{equation}

We apply this to  $U_\tau$ as constructed in corollary~\ref{ONCR} which satisfies the orthogonality constraint \eqref{ONOPT2}, and we obtain
\begin{equation}\label{onofineqRT3}
\frac{1}{4}\int_{S^2}|\nabla U_\tau|^2 \d \sigma - \ln\left(\int_{S^2}e^{U_\tau} \d \sigma\right) + \int_{S^2} U_\tau \d \sigma \geq \frac{1}{8}\int_{S^2}|\nabla U_\tau |^2 \d \sigma
\end{equation}

By the definition of $U_\tau = u\circ \tau + \log(\mathcal{J}_\tau)$, and by the chain rule  $\log(\mathcal{J}_{\tau^{-1}}\circ \tau) = -\log(\mathcal{J}_\tau)$
so that
$$
\int_{S^2}|\nabla U_\tau |^2 \d \sigma = \int_{S^2}|\nabla u\circ \tau -  \nabla \log(\mathcal{J}_{\tau^{-1}}\circ \tau)|^2 \d \sigma =
 \int_{S^2}|\nabla u -  \nabla \log(\mathcal{J}_{\tau^{-1}})|^2\d \sigma\ ,
$$
where the last equality is the usual conformal invariance of  $f\mapsto \int_{S^2}|\nabla f|^2\d \sigma$. 
Combining this with the conformal invariance of the Onofri functional, we can rewrite \eqref{onofineqRT3} as 
$$\frac{1}{4}\int_{S^2}|\nabla u|^2 \d \sigma - \ln\left(\int_{S^2}e^u\d \sigma\right) + \int_{S^2} u \d \sigma \geq 
\frac18 \int_{S^2}|\nabla u -  \nabla \log(\mathcal{J}_{\tau^{-1}})|^2\d \sigma\ .$$
Since $\log(\mathcal{J}_{\tau^{-1}})\in \mathcal{M}$, this proves \eqref{ONSTX1}. 
Next, using the Onofri inequality itself, written in the form \eqref{onofineqEV}, on the right side of \eqref{onofineqRT3} yields
\begin{equation*}
\frac{1}{4}\int_{S^2}|\nabla U_\tau|^2 \d \sigma - \ln\left(\int_{S^2}e^{U_\tau} \d \sigma\right) + \int_{S^2} U_\tau \d \sigma \geq \frac{1}{2}H(e^{1|U_\tau} )\ .
\end{equation*}
This is conformally equivalent to \eqref{ONSTX1B}.
\end{proof}

\begin{cor}\label{L1STAB} Let $u$ be a  function on $S^2$ such that $\int_{S^2}e^{u}\d \sigma = 1$ and $\int_{S^2}|\nabla u|^2\d \sigma < \infty$. Then
\begin{equation*}
\frac{1}{4}\int_{S^2}|\nabla u|^2 \d \sigma - \ln\left(\int_{S^2}e^u\d \sigma\right) + \int_{S^2} u \d \sigma \geq 
\frac{1}{4} \inf_{v\in \mathcal{M}_O} \left\{\ \|e^u - e^v\|_1^2\ \right\}\ .
\end{equation*}
\end{cor}

\begin{proof} Simply apply Pinsker's inequality \eqref{Pinsker} to \eqref{ONSTX1B}.
\end{proof}

\section{The log HLS Inequality on $S^2$}

Let $E$ denote the space of real valued functions in  $L^\infty(S^2)$ modulo the constants, so that for $u\in E$, 
$\|u\|_E = \inf_{c\in \R}\|u -c\|_\infty$.  Let $E^*$ denote the subspace of $L^1(S^2)$ consisting of real valued functions 
$f$ such that  $\int_{S^2} f\d \sigma = 0$. With the bilinear form $\langle u,f\rangle := \int_{S^2}u f \d \sigma$, $E$ and $E^*$ are a dual pair of Banach spaces.

For $u\in E$, define
$$
\mathcal{F}(u) = \begin{cases} {\displaystyle \frac14\int_{S^2}|\nabla u|^2\d \sigma} & u\in H^1(S^2)\\
\phantom{E}\quad +\infty & u\notin H^1(S^2)\ .
\end{cases}
$$
It is easy to see that $\mathcal{F}$ is a proper lower semicontinuous convex function on $E$.

Define the kernel $G(\omega,\omega')$ on $S^2\times S^2$ by
$$
G(\omega,\omega') :=-2\log (\omega - \omega')\ .
$$
It is easy to see that 
$$f\mapsto 2\int_{S^2}\int_{S^2} f(\omega)G(\omega,\omega')f(\omega)\d \sigma \d \sigma =: \mathcal{F}^*(f)$$ 
is well defined on $E$, with values in $(-\infty,\infty]$.
In fact, $G$ is the Greens function for $-\Delta$ on $S^2$:  For $f\in L^2(S^2)$ with $\int_{S^2}f\d \sigma  =0 $, define 
$Gf(\omega) := \int_{S^2}G(\omega, \omega')f(\omega')\d \sigma$. Then $-\Delta Gf = f$. 
Since $-\Delta \geq 2 I$, on the orthogonal complement of the constants in $L^2(S^2)$, it is not hard to show that
$\mathcal{F}$ is a proper  lower semicontinuous strictly convex function on $E$. 

In fact, as the notation suggests, $\mathcal{F}^*$ is the Legendre transform of $\mathcal{F}$, {\em vice-versa}.  To see that 
 $\mathcal{F}^*$ is the Legendre transform of $\mathcal{F}$, compute

\begin{eqnarray*}
\langle u,f\rangle - \mathcal{F}(u)  &=& \langle u,f\rangle - \frac14 \langle (-\Delta)^{1/2}u, (-\Delta)^{1/2}u \rangle \\
&=& 2\langle (-\Delta)^{-1/2}f,(-\Delta)^{-1/2}f\rangle  - \frac14 \left\|(-\Delta)^{1/2}u - 2(-\Delta)^{-1/2}f\right\|_{L^2(S^2)}\\
&\leq& 2\langle f, (-\Delta)^{-1} f\rangle = \mathcal{F}^*(f)\ ,
\end{eqnarray*} 
and there is equality for $u = -2\Delta^{-1} f$. 
Then since both $\mathcal{F}$ and $\mathcal{F}^*$ are proper lower semicontinuous convex functions, the fact that the Legendre transform of 
$\mathcal{F}^*$ is a consequence of the Fenchel-Moreau Theorem, though it is of course easy to prove directly by completing a square, as above; see \cite{CL92}. 

Next, define the functional $\mathcal{E}$ on $E$ by 
$$
\mathcal{E}(u) = \log\left(\int_{S^2}e^u \d \sigma\right) - \int_{S^2} u \d \sigma\ .
$$
For $f\in E^*$ and $u\in E$, 
\begin{eqnarray*}
\langle u,f\rangle - \mathcal{E}(u)   = \langle u,f+1 \rangle - \log\left(\int_{S^2}e^u \d \sigma\right)\ .
\end{eqnarray*}
Suppose $f+1 < 0$ on a set $A$ of positive measure $|A|$. Let $u := a 1_A$ for $a < 0$. Then
\begin{eqnarray*}
 \langle u,f+1 \rangle - \log\left(\int_{S^2}e^u \d \sigma\right) &=& a\int_{S^2} g\d \sigma -\log\left( \int_{S^2} e^{a1_A} \right)\\
 &=& a\int_{S^2} g\d \sigma -\log\left( e^a|A| + (1-|A|) \right)\\
 &\geq&  a\int_{S^2} g\d \sigma -\log\left(1-|A| \right)\ .
\end{eqnarray*}
Taking the limit $a\mapsto -\infty$, we obtain $\mathcal{E}^*(f) = +\infty$. On the other hand, if $f+1\geq 0$, since $\int_{S^2}(f+1)\d \sigma = 1$, 
we have from \eqref{DV2} that
$$
\sup_{u\in L^\infty(S^2)}\left\{ \int_{S^2} u(f+1)\d \sigma - \log\left(\int_{S^2} e^u \d \sigma\right) \right\} = \begin{cases} \int_{S^2} (f+1)\log (f+1)\d \sigma & \int_{S^2}f\d \sigma =0\\
\phantom{X}\quad +\infty &  \int_{S^2}f\d \sigma \neq 0\end{cases}\ .
$$
That is
$$
\mathcal{E}^*(f) = 
\begin{cases} \int_{S^2} (f+1)\log (f+1)\d \sigma & \int_{S^2}f\d \sigma =\ ,\ f+1 \geq 0\\
\phantom{X}\quad +\infty &  \phantom{X}\quad{\rm otherwise}\end{cases}\ .
$$

Since the functional over which the supremum is taken does not change when a constant is added to $u$, it is a functional on $E$, and we may take the supremum over $E$ instead.  For the same reason, we may suppose that $\int_{S^2}e^u= 1$, and then 
$$\mathcal{E}^*(f) =  \sup_{u\in L^\infty(S^2)}\left\{ \int_{S^2} u(f+1)\d \sigma \ :\   \int_{S^2}e^u= 1\ \right\}$$
However,
$$
\mathcal{E}^*(f)  - \int_{S^2} u(f+1)\d \sigma = \int_{S^2} \left( \frac{f+1}{e^u}\right)\log \left( \frac{f+1}{e^u}\right) e^u \d \sigma\ ,
$$
and by Jensen's inequality, this is non negative and equals zero only for $u = \log (f+1)$. 

This shows that when  $\mathcal{E}(u) < \infty$ and   $\mathcal{E}^*(f) < \infty$, we have equality in Young's Inequality; that is, 
$$
\langle u,f\rangle =  \mathcal{E}(u) + \mathcal{E}^*(f)
$$
if and only if $u = \log (f+1)$. 
It follows that $\partial \mathcal{E}(u) = \{\log(f+1)\}$ and $\partial \mathcal{E}^*(f) = \{e^u-1\}$.

Now the Onofri inequality can be written as $\mathcal{E} \leq \mathcal{F}$. The dual inequality is $\mathcal{E}^* \geq \mathcal{F}^*$
is
\begin{equation*}
\int_{S^2} (f+1)\log (f+1)\d \sigma \geq -2\int_{S^2}\int_{S^2} f(\omega)\log|\omega - \omega'| f(\omega')\d\sigma \d \sigma\ ,
\end{equation*}
valid for all $f\in L^1(S^2)$ such that $\int_{S^2}f\d \sigma =0$ and $f+1 \geq 0$. 

Define the spherical ``Helmholz free energy'' $\mathcal{H}_S$ on $E^*$ by
\begin{equation*}
\mathcal{H}_S(f) = \int_{S^2} (f+1)\log (f+1)\d \sigma +2\int_{S^2}\int_{S^2} f(\omega)\log|\omega - \omega'| f(\omega')\d\sigma \d \sigma 
\end{equation*} 
in case $f+1 \geq 0$, and $\mathcal{H}_S(f) = \infty$ otherwise.  The spherical log HLS inequality is $\mathcal{H}_S(f) \geq 0$, and by Theorem~\ref{dualsharp}, there is equality if and only if $f+1 = e^u$ with $u\in \mathcal{M}_O$, where $\mathcal{M}_O$ is specified in \eqref{Mspec}. 
We now prove stability for this inequality.

\begin{thm} For all $f\in E$, $f+1\geq 0$, 
\begin{equation*}
\mathcal{H}_S(f) \geq \frac{1}{8}\inf_{v\in \mathcal{M}}\left\{\|(f+1) - g\|_1^2 \right\}\ .
\end{equation*}
\end{thm}

\begin{proof} By Theorem~\ref{STRYO},
\begin{equation}\label{pr1X}
\cE(u) + \cE^*(f) \geq \langle u,f\rangle +  \frac12\| f - (e^u-1)\|_1^2 \ .
\end{equation}

Corollary~\ref{L1STAB} provides the stability inequality
\begin{equation}\label{pr2X}
\cF(u) - \cE(u) \geq  \frac{1}{4} \|e^u - e^{u_0}\|_1^2\ .
\end{equation}
Summing both sides of \eqref{pr1X} and  \eqref{pr2X}, rearranging terms, and using the fact that  $\cF^*(f) \geq \langle u,f\rangle - \cF(u)$ for all $x$, yields
\begin{eqnarray*}
\cE^*(f) - \cF^*(f)  &\geq&  \frac{1}{4}\left( \| e^u- e^{u_0} \|_1^2 +  \| (f+1)- e^u \|_1^2\right)\nonumber \\
&\geq&  \frac{1}{8}\left( \| e^u- e^{u_0} \|_1 +  \| (f+1)- e^u \|_1\right)^2\nonumber \\
&\geq&  \frac{1}{8} \| (f+1)- e^{u_0} \|_1^2 \ . 
\end{eqnarray*}
\end{proof}

\section{Stability for the log HLS inequality on $\R^2$}

The stability result proved in the last section for the  log HLS inequality on the sphere transfers directly to log HLS inequality on the plane through the stereographic projection $\mathcal{S}$. As observed by Lieb,  for $x,x'\in \R^2$
\begin{equation*}
|\mathcal{S}^{-1}(x) - \mathcal{S}^{-1}(x') |^2 = \frac{4}{(1+|x|^2)(1+|x'|^2)}|x-x'|^2\ .
\end{equation*}
This may be verified by a  straightforward computation, but it is the key to the conformal invariance of the HLS inequality that Lieb discovered. 
It is then also the key to the conformal invariance of the log HLS inequality.

As before, let $h(x) = \left(\frac{2}{1+|x|}\right)^2$ be the Jacobian of  $\mathcal{S}$.  Given a function 
$\varphi \in L^1(\R^2)$, define the function $T\varphi$ by
$$
T\varphi := \pi(\varphi/h)\circ \mathcal{S}. 
$$
Since $h>0$, 
$$\int_{S^2} |T\varphi|\d \sigma = \pi\int_{\R^2} (|\varphi(x)|/h(x))\d \mu = \int_{\R^2}|\varphi|\d x \ .$$
 It follows that the map $T: \varphi \mapsto T\varphi$ is an isometry from $L^1(\R^2)$ onto $L^1(S^2)$, and hence $T$ preserves 
integrals, and again because $h>0$, $T$ preserves positivity. 
In particular, if $\varphi \geq 0$ and  $\int_{\R^2}\varphi \d x =1$, then  $T\varphi \geq 0$ and $\int_{S^2}(T\varphi -1)\d \sigma = 0$.  We are now ready to prove Theorem~\ref{main}

\begin{proof}[Proof of Theorem~\ref{main}]  Consider any $\rho\geq 0$ on $\R^2$ such that $\int_{\R^2} |\rho(x)\log(e+|x|^2)\d x < \infty$ and such that $\int_{\R^2} \rho \d x = 1$. 
Define a function $f$ on $S^2$ by  $f := T\rho -1$. Then by what has been observed just above, $f\in L^1(S^2)$, 
$\int_{S^2}f\d \sigma =0$, and $f+1\geq 0$.

Then with $f$ and $\varphi$ related in this way, one may check that
\begin{multline*}
\int_{S^2}(f+1)\log(f+1)\d \sigma + 2\int_{S^2}\int_{S^2}f(\omega)\log|\omega - \omega'| f(\omega')\d \sigma \d \sigma = \\
\int_{\R^2}\varphi \log \varphi \d \sigma + 2\int_{\R^2}\int_{\R^2}\varphi(x)\log|x - x'| \varphi(x')\d x \d x' + (1+ \log \pi)\ . \\
\end{multline*} 
Finally $\|(f+1)- e^{u_0} \|_1 = \|\rho - T^{-1}e^{u_0}\|_1$, and $T^{-1}e^{u_0}$  is an optimizer for the Euclidean  log HLS inequality. 
Therefore, 
$$
\int_{\R^2}\rho \log \rho \d \sigma + 2\int_{\R^2}\int_{\R^2}\rho(x)\log|x - x'| \rho(x')\d x \d x' + (1+ \log \pi)  \geq \frac{1}{8}\inf_{\psi\in \mathcal{M}}\{ \|\rho - \psi\|_1^2\}\ .
$$
\end{proof}

\section{The one dimensional case}

After the first version of this work was complete, Rupert Frank pointed out to me  that the analog on $S^1$  of the result of Gui and Moradifam on $S^2$ 
had been proved by Osgood, Philips and Sarnak \cite{OPS}, and noted that it could be employed in the same manner to obtain stability results on 
$S^1$. In this section, we briefly develop this remark. 

 To state the inequality of \cite{OPS}, let ${\rm d}\sigma$ denote the normalized uniform measure on $S^1$. 
For continuous $u$ on $S^1$, define $\|(-\Delta)^{1/4}u\|_2^2 := \sum_{k\in \Z} |k| |\widehat{u}(k)|^2$ where $\widehat{u}(k)$ is the $k$th Fourier coefficient of $u$, and then extend this definition by continuity to obtain the Sobolev space $H^{1/2}(S^1)$. The inequality of \cite{OPS} is that all  function $u$ on $S^1$
 that satisfy 
 the constraint
\begin{equation}\label{onofineqRTCZ}
\int_{S_1}e^{u(\omega)} \omega\cdot{\bf n}{\rm d}\sigma = 0 \quad{\rm for\ all}\ {\bf n}\in S^1
\end{equation}
satisfy
\begin{equation}\label{onofineqRTZ}
\frac 14 \int_{S^1}|(-\Delta)^{1/4}u|^2{\rm d}\sigma \geq \log\left(\int_{S^1}e^u {\rm d}\sigma\right) - \int_{S_1}u{\rm d}\sigma\ .
\end{equation}
These are the analogs for $S^1$ of \eqref{onofineqRTC} and \eqref{onofineqRT} for $S^2$. As Widdom noted \cite{W88}, 
the constrained inequality proved in \cite{OPS} is one a a family of  constrained inequalities that may be obtained very simply from results of Szeg\"o \cite{GS58} on Toeplitz forms.

If the constant $\tfrac14$ in \eqref{onofineqRTZ} is replaced by 
$\tfrac12$,  yielding
\begin{equation}\label{onofineqRTZ2}
\frac 12 \int_{S^1}|(-\Delta)^{1/4}u|^2{\rm d}\sigma \geq \log\left(\int_{S^1}e^u {\rm d}\sigma\right) - \int_{S_1}u{\rm d}\sigma\ ,
\end{equation}
the inequality is valid, and sharp,  without the constraint \eqref{onofineqRTCZ}.  This is an inequality 
of Lebedev and Milin (see \cite{OPS}), and a proof by competing symmetries, with determination of the cases of equality, 
may be found in \cite{CL92}. That is, the Lebedev-Milin inequality is the $S^1$ analog of the Onofri inequality on $S^2$.  
Exactly as explained in Section 4,
\eqref{onofineqRTZ2} may be written in terms of relative entropy as
\begin{equation}\label{onofineqRTZ6}
 \int_{S^1}|(-\Delta)^{1/4}u|^2{\rm d}\sigma \geq  2 H(1||e^u)\ .
\end{equation}
The inequality \eqref{onofineqRTZ} then yields 
\begin{equation}\label{onofineqRTZ5}
\frac 12 \int_{S^1}|(-\Delta)^{1/4}u|^2{\rm d}\sigma - \log\left(\int_{S^1}e^u {\rm d}\sigma\right) +  \int_{S_1}u{\rm d}\sigma \geq 
\frac 14 \int_{S^1}|(-\Delta)^{1/4}u|^2{\rm d}\sigma\ 
\end{equation} 
for all $u$ satisfying the constraint \eqref{onofineqRTCZ}.
It is a simple matter to adapt the proof of Corollary~\ref{ONCR} to show that for all functions $u$ on $S^1$ such that
$\int_{S^1}e^{u}\d \sigma = 1$ and $\|(-\Delta)^{1/4} u|\|_2< \infty$,
 there exists a conformal transformation $\tau$ on $S^1$ (see \cite{CL92})  such that with $U_\tau = u\circ \tau + \log(J_\tau)$,  
\begin{equation}\label{ONOPT2Z}
\int_{S^1}e^{U_\tau(\omega)} \omega \cdot {\bf n}\d \sigma = 0
\end{equation}
for all ${\bf n}\in S^1$.  Applying the inequality \eqref{onofineqRTZ}  followed by \eqref{onofineqRTZ6} yields
\begin{eqnarray}\label{onofineqRT3Z}
\frac{1}{2}\int_{S^1}|(-\Delta)^{1/4} U_\tau|^2 \d \sigma - \ln\left(\int_{S^1}e^{U_\tau} \d \sigma\right) +
 \int_{S^1} U_\tau \d \sigma &\geq& \frac{1}{4}\int_{S^1}|(-\Delta)^{1/4}  U_\tau |^2 \d \sigma\nonumber\\
 &\geq& \frac12H(1||e^{U_\tau})\ .
\end{eqnarray}
Proceeding from here exactly as in Section 4 yields a stability bound for the inequality of \cite{OPS}:
\begin{cor}\label{L1STABS1} Let $u$ be a  function on $S^1$ such that $\int_{S^2}e^{u}\d \sigma = 1$ and $\int_{S^2}|(-\Delta)^{1/4} u|^2\d \sigma < \infty$. Then
\begin{equation*}
\frac{1}{2}\int_{S^1}|(-\Delta)^{1/4} u|^2 \d \sigma - \ln\left(\int_{S^1}e^u\d \sigma\right) + \int_{S^1} u \d \sigma \geq 
\frac{1}{4} \inf_{v\in \mathcal{M}_O} \left\{\ \|e^u - e^v\|_1^2\ \right\}\ .
\end{equation*}
\end{cor}
One can also prove an anlog of \eqref{ONSTX1} if desired,  but Corllary~\ref{L1STABS1} is what is relevant to stability for the log HLS inequality on $S^1$. 

The sharp logarithmic HLS inequality in one dimension, both on $S^1$ and $\R$, is proved as a special case of \cite[Theorem 1]{CL92}, and is related there to  Lebedev-Milin inequality by Legendre transforms. Proceeding from Corollary~\ref{L1STABS1} exactly as in this paper leads to the one dimensional version of Theorem~\ref{main}, in which the constant is now $\tfrac{1}{4}$ in place of $\tfrac{1}{8}$.

\medskip

\noindent{\bf Acknowledgements} I am grateful for useful conversations with Jean Dolbeault, Rupert Frank and Michael Loss. In particular, 
I am grateful to Jean Dolbeault for bringing the work of Gui and Moradifam 
\cite{GM18} to my attention and for suggesting to me that it could be the source of a 
stability inequality for Onofri's inequality. Though such a stability result has been recently proved in a more general setting by 
Chen, Lu and Tang \cite{CLT23}, the
more explicit result obtained starting from \cite{GM18} is well-adapted to our purpose.  I had originally planned to prove a stability 
result for Onofri's inequality using the Bianchi-Egnell approach \cite{BE91}, using the recent advances in \cite{DEFFL} that yield explicit constants.
In fact, I had earlier begun an investigation of stability for Onofri's inequality with Francis Seuffert, a former student who wrote a thesis on stability 
problems \cite{S17}, some years ago when he was a postdoc. 
We had arrived at an explicit ``local bound'', that is, a stability estimate near the manifold of optimizers. It seems quite certain one could obtain an 
explicit stability bound from here using the methods of \cite{DEFFL}. However, the result of \cite{CLT23} eliminates the main motivation for 
carrying this out, and the result of  Gui and Moradifam leads to a very clean result.


\begin{thebibliography}{99}

\footnotesize{

\bibitem{A79} T.~Aubin, \textit{Meilleures constantes dans le th\'eor\`eme d’inclusion de Sobolev et un th\'eor\`eme de Fredholm 
non lin\'eaire pour la transformation conforme de la courbure scalaire}, J. Funct. Analysis. {\bf 32} (1979), no.2, 148--174.

\bibitem{BCL}  K.~Ball,  E.~A.~Carlen  and  E.~H.~Lieb, \textit{Sharp uniform convexity and smoothness inequalities for trace norms}. Invent. Math. \textbf{115} (1994), 463--482.


\bibitem{B93} W.~Beckner, \textit{Sharp Sobolev inequalities on the sphere and the Moser-Trudinger inequality}, Ann. of Math. {\bf 138} (1993) 213--242. 

\bibitem{Cs67} I. Csisz\'ar, “Information-type measures of difference of probability distributions and indirect observations,” Studia Sci. Math. Hungar., {\bf 2} (1967), 299--318.

\bibitem{BE91}
G.~Bianchi and H.~ Egnell, 
\textit{A note on the Sobolev inequality}, J. Func. Analysis, {\bf 100} (1991), 18--24. 

\bibitem{BCC}
A.~Blanchet, E.~A.~Carlen, J.~A.~Carrillo  \textit{Functional
inequalities, thick tails and asymptotics for the critical mass
Patlak-Keller-Segel model},  Jour. Func. Analysis   {\bf 262} (2010), 2142--2230.

\bibitem{BDP06} A.~Blanchet, J.~Dolbeault, and B.~Perthame, \textit{Two-dimensional 
Keller-Segel model}: optimal critical mass and qualitative properties of the 
solutions. Electron. J. Differential Equations, {\bf 44} (2006) 1--33

\bibitem{C11} E.~A.~Carlen, \textit{Functional inequalities and dynamics}, pp. 17--85 in \textit{Nonlinear PDE’s and Applications
C.I.M.E. Summer School, Cetraro, Italy 2008}, Editors: Luigi Ambrosio, Giuseppe Savar\'e, Lecture Notes in Math., {\bf 2028}, Springer, Heidelberg, 2011.

\bibitem{C16} E.~A.~Carlen, \textit{Duality and stability for functional inequalities}
Annales de la Facult\'e des sciences de Toulouse: Math\'matiques {\bf 26} (2016), 319--350.

\bibitem{CCL10}
E.~A.~Carlen, J.~A.~Carrillo and M.~Loss, \textit{Hardy-Littlewood-Sobolev Inequalities via Fast Diffusion Flows}, 
P.N.A.S. {\bf 107} (2010)  19696--19701.

\bibitem{CF} E.~A.~Carlen and A.~Figalli, \textit{Stability for a GNS inequality and the log-HLS inequality, with application to the critical mass Keller–Segel equation},  Duke Math. Jour., {\bf 162} (2013), 579--625.

\bibitem{CL90} E.~A.~Carlen and M.~Loss, \textit{Extremals of functionals with competing symmetries},
Jour. of Func. Analysis {\bf 88} (1990), 437--456.


\bibitem{CL92} E.~A.~Carlen and M.~Loss,   \textit{Competing symmetries, the logarithmic HLS inequality and 
Onofri's inequality on $S^n$} , Geom. and Funct. Anal. {\bf 2} (1992)  90--104.  


\bibitem{CL88} E.~A.~Carlen and M.~Loss, \textit{Competing symmetries of some functionals arising in
mathematical physics}  277--288 in \textit{Stochastic processes, physics and geometry}, Proceedings of the
International Conference on Stochastic Processes, Physics and Geometry, Ascona/Locarno, Switzerland, 4 – 9 July 1988, eds. S.~Albeverio, G.~Casati, U.~Cattaneo, D.~Merlini and  R.~Moresi , World Sci. Publ., Teaneck, NJ, 1990. 

\bibitem{CY88} A.~Chang and P.~C.~Yang, \textit{Conformal deformation of metrics on $S^2$}, J. Differential Geom., {\bf 27} (1988), 259--296.

\bibitem{CLT23} L.~Chen, G.~Lu, and H.~Tang, \textit{Sharp stability of log-Sobolev and Moser-Onofri inequalities on the sphere}, Jour. Func. Anal., {\bf 285} (2023), 110022.

\bibitem{DD10}
M.~Del Pino and J.~Dolbeault,  \textit{Best constants for
Gagliardo-Nirenberg inequalities and applications to nonlinear
diffusions}, Jour. Math. Pures Appl., {\bf 81} (2002), 301--342.

\bibitem{DV83} M.~D.~Donsker and S.~R.~S.~Varadhan, \textit{Asymptotic evaluation of certain Markov process expectations for large time. IV}. Comm.  Pure and Appl. Math., {\bf 36} (1983) 183--212.

\bibitem{DEFFL} J.~Dolbeault, M.~J.~Esteban, A.~Figalli, R.~L. Frank, M.~Loss, \textit{Sharp stability for Sobolev and log-Sobolev inequalities, with optimal dimensional dependence}, arXiv preprint 2209.08651 (2022).

\bibitem{DP04}
J.~Dolbeault and B.~Perthame, 
\textit{Optimal critical mass in the two-dimensional Keller-Segel model in $\R^2$}, C. R. Math. Acad. Sci. Paris, {\bf 339} (20024), 611--616.

\bibitem{GS58} U.~Greenlander and G.~Szeg\"o, \textit{Toeplitz forms and their application}, Berkeley, Univ. of California Press, 1958.

\bibitem{GM18} C.~Gui and A.~Moradifam, \textit{The Sphere Covering Inequality and Its Applications}, Invent. Math. {\bf 214} (2018), 1169--1204.

\bibitem{K67}\ S.~Kullback \textit{ Lower bound for discrimination information in terms of variation},
IEEE Transactions on Information Theory, {\bf 13} 126--127 (1967).



\bibitem{L83} E.~H.~Lieb, \textit {Sharp constants in the Hardy-Littlewood-Sobolev and related inequalities},
Ann. of Math. {\bf 118} (1983), 349--374.

\bibitem{M71} J.~Moser, \textit{A sharp form of an inequality by N. Trudinger}, Indiana U. Math. J. {\bf 20} (1971), p. 1077--1091.

\bibitem{O82} E.~Onofri, \textit{On the positivity of the effective action in a theory of random surfaces}, Comm. Math. Phys. {\bf 86} (1982), no.3, 321--326.

\bibitem{OV82} E.~Onofri and M.~A.~Virasoro, \textit{On a formulation of Polyakov's string theory with regular classical solutions}, 
Nuclear Physics B, {\bf 201} (1982), 159--175.

\bibitem{OPS} B.~Osgood, R.~Philips and P.~Sarnak, \textit{Extremal determinants of Laplacians}, Jour. Func. Analysis, {\bf 80} (1988), 148--211.

 \bibitem{P64} M.~S.~Pinsker, \textit{Information and Information Stability of Random Variables and Processes}. Holden-Day  (1964).
 
 \bibitem{P81} A.~M.~Polykov, \textit{Quantum theory of Bosonic strings} Phys. Lett. B, {\bf 103} (1981) 207--210.
  
\bibitem{S17} F.~Seuffert, \textit{An extension of the Bianchi–Egnell stability estimate to Bakry, Gentil, and Ledoux's generalization of the Sobolev inequality to continuous dimensions}, Jour. func. Anal. {\bf 273} (2017), 3094--3149

\bibitem{V84} S.~R.~S.~Varadhan, \textit{Large Deviations and Applications}, CBMS-NSF Regional Conference Series in 
Applied Mathematics} {\bf 46}, S.I.A.M.,
Philadelphia, 1984.

\bibitem{W88} H.~Widdom, \textit{On an inequality of Osgood, Philips and Sarnak}, Proc. Am. Math. Soc. {\bf 102} (1988), 773-774.


\end{thebibliography}
\end{document}